\documentclass[12pt, reqno]{article}
\usepackage{fullpage}
\usepackage{psfrag}
\usepackage[parfill]{parskip}
\usepackage{float}
\usepackage{graphicx}
\usepackage{epsfig}
\usepackage{amsmath, amsthm, amssymb}
\usepackage{verbatim}
\usepackage{multicol}
\usepackage{url}
\usepackage{latexsym}
\usepackage{mathtools}
\usepackage{mathrsfs}
\usepackage[colorlinks = true,
            linkcolor  = blue,
            citecolor  =blue,
            urlcolor   =blue,
            bookmarks  = true]{hyperref}
 \usepackage{enumitem}
\setlist[enumerate,1]{label=(\roman*)}
\usepackage[normalem]{ulem}
\usepackage{bm}
\usepackage{dsfont}
\usepackage{stmaryrd}
\usepackage[dvipsnames]{xcolor}

\usepackage[backend=biber, style=numeric, citestyle=numeric-comp, giveninits=true, doi=false, isbn=false, url=false, eprint=false, maxbibnames=99]{biblatex}
\DeclareFieldFormat{labelnumberwidth}{\mkbibbrackets{#1}}
\DeclareNameAlias{author}{given-family}
\DeclareNameAlias{editor}{given-family}

\DeclareFieldFormat[article,inproceedings,incollection,misc,unpublished,book]{title}{#1}
\DeclareFieldFormat[article]{journaltitle}{#1}
\DeclareFieldFormat[article]{volume}{#1}

\AtEveryBibitem{%
  \clearfield{issn}
  \clearfield{url}
  \clearfield{eprint}
  \clearfield{doi}
  \clearfield{isbn}
  \clearfield{number}
  \clearfield{mrnumber}
  \clearfield{mrclass}
  \clearfield{reviewer}
}

\renewbibmacro{in:}{}

\renewbibmacro*{volume+number+eid}{%
  \printfield{volume}%
  \setunit{\addcomma\space}%
  \printfield{eid}}

\DeclareDelimFormat{namedelim}{\addcomma\space}

\DeclareFieldFormat{pages}{#1}

\numberwithin{equation}{section}

\theoremstyle{plain}

\newtheorem{theorem}{Theorem}[section]
\newtheorem{lemma}[theorem]{Lemma}
\newtheorem{proposition}[theorem]{Proposition}

\theoremstyle{definition}

\newtheorem{definition}[theorem]{Definition}

\newtheorem{remark}[theorem]{Remark}

\DeclarePairedDelimiterX{\ip}[2]{\langle}{\rangle}{#1, #2}
\DeclarePairedDelimiterX{\norm}[1]{\lVert}{\rVert}{#1}

\newcommand{\bS}{\mathbb{S}}
\newcommand{\R}{\mathbb{R}}

\usepackage{cleveref}
  \crefname{theorem}{Theorem}{Theorems}
  \crefname{thm}{Theorem}{Theorems}
  \crefname{thm*}{Theorem*}{Theorems}
  \crefname{lemma}{Lemma}{Lemmas}
  \crefname{lem}{Lemma}{Lemmas}
  \crefname{remark}{Remark}{Remarks}
  \crefname{prop}{Proposition}{Propositions}
\crefname{notation}{Notation}{Notations}
\crefname{claim}{Claim}{Claims}
  \crefname{defn}{Definition}{Definitions}
  \crefname{corollary}{Corollary}{Corollaries}
  \crefname{section}{Section}{Sections}
  \crefname{figure}{Figure}{Figures}
    \crefname{assumption}{Assumption}{Assumptions}

\begin{document}

\title{An upper bound on the smallest singular value of dense signed random combinatorial matrices}

\author{\textsc{Yuxiao Shi, Zeyan Song, Kexin Yu, and Tingzhou Yu}}

\date{}

\maketitle


\begin{abstract}
Let $M_n$ be an $n\times n$ random matrix with entries in $\{-1, 0, 1\}$, where each row is independently and uniformly sampled from the set of all vectors in $\{-1, 0, 1\}^n$ containing exactly $d$ nonzero entries, with $d=pn$ for some fixed constant $p\in(0,1/2]$. A recent result by one of the authors gives a lower bound of order $n^{-1/2}$ for the smallest singular value $s_n(M_n)$ of this model with high probability. In this paper, we establish a matching upper bound for $s_n(M_n)$. Specifically, there exists a constant $C_p>0$ depending only on $p$ such that
\[
\forall \varepsilon >0 \qquad \mathbb{P} \left( s_n(M_n) \le \frac{\sqrt{d}}{\varepsilon^2 n} \right)
\ge  1-C_p\left(\varepsilon + \frac{1}{\sqrt{d}}\right).
\]
Consequently, in the dense regime, the smallest singular value $s_n(M_n)$ of signed random combinatorial matrices is typically of order $n^{-1/2}$. 
\end{abstract}

\bigskip

{\small
}

\section{Introduction}\label{sec1}
Let $A$ be an $n \times n$ real matrix, and let $s_1(A)\ge \cdots \ge s_n(A)\ge 0$ be its singular values. The smallest singular value can be expressed as 
\[
s_n(A)=\inf_{x \in \bS^{n-1}} \|Ax\|_2,
\]
where $\bS^{n-1}$ denotes the unit sphere in $\R^n$. The quantity \(s_{n}(A)\) measures the distance from $A$ to the set of singular matrices in the operator norm. In particular,  $s_{n}(A)>0$ if and only if $A$ is invertible, and then $s_{n}(A)=1/\|A^{-1}\|$, where $\|A\|:=\sup_{x\in \mathbb S^{n-1}}\|Ax\|_2$ is the operator norm of the matrix $A:\ell_2^n \to \ell_2^n$. 

Quantitative estimates for $s_n(A)$ are central in random matrix theory and numerical linear algebra. In numerical analysis, the study of the smallest singular value is closely linked to the condition number  $\kappa(A)=s_1(A)/s_n(A)$, which measures the worst-case sensitivity of linear systems involving $A$ to perturbations \cite{spielman2004smoothed,tikhomirov2022quantitative, meyer2023matrix, matrixcomputation}. 
Beyond numerical stability, the smallest singular value also enters the study of singularity probability \cite{basak2021sharp,jaindiscrete,
litvak2022singularity,tikhomirov2020singularity}, empirical spectral distribution of random matrices \cite{TVuniversality, bordenave2012around, RTcirclawmin}, and asymptotic geometric analysis \cite{litvak2005smallest}.

For random matrices with independent entries, the behavior of the smallest singular value is now well understood in many classical settings. 
The study of the smallest singular value dates back to von Neumann and his collaborators, who predicted that the smallest singular value is typically of order $n^{-1/2}$ \cite{collectedworks,numericalinverting}. This scale was confirmed in the Gaussian case by Edelman  \cite{eigenvalues} and Szarek  \cite{conditionnumbers}, and later extended to independent and identically distributed (i.i.d.) subgaussian matrices by Rudelson and Vershynin \cite{thelittle,theleast}. Further developments cover weaker moment assumptions and related rectangular models see, for example, \cite{coveringof, anuppersquare, kahn1995probability,  litvak2005smallest, rudelson2008invertibility, rudelson2009smallest, mendelson2014singular, livshyts2021smallestTK,  tao2009inverse, tao2010random}. 


The situation is substantially more delicate when dependencies are present among the matrix entries, and progress has been comparatively slower. For instance, a sharp lower bound estimate for the smallest singular value of $n\times n$ symmetric random matrices, where the upper triangular entries are i.i.d.\ subgaussian random variables, was only recently proved by Campos, Jenssen, Michelen, and Sahasrabudhe \cite{symmetricsubgaussian}, while the analogous non-symmetric result was established much earlier by Rudelson and Vershynin \cite{thelittle}. Apart from symmetric matrices, a class of combinatorial matrix models has recently attracted considerable attention, primarily due to their close connections with graph theory (see, for example, \cite{cookregular,jainapprox,jainregular,litvakdigraphs,litvakshiftedregular,tran}).

In this paper, we study such a combinatorial model with random signs. Let $1\le d\le n$, and let $M_n$ be the random $n\times n$ matrix whose rows are independent and uniformly distributed over the set of vectors in $\{-1, 0, 1\}^n$ containing exactly $d$ nonzero entries. We refer to this as a signed random combinatorial matrix: the support of each row has fixed size $d$, and the nonzero entries have signs. The signed random combinatorial matrix is more tractable than the unsigned one because of the following support and sign representation. Let $Q_n$ be the $n\times n$ (unsigned) random combinatorial matrix whose rows are independent and uniformly distributed over $\{0,1\}^n$ with exactly $d$ ones, and let $\Xi$ be an independent sign matrix with i.i.d.\ Rademacher entries. Then \[M_n\stackrel{d}{=} Q_n\circ  \Xi,\]where $\circ$ denotes the Hadamard product of two matrices $A$ and $B$ such that $( A \circ  B)_{ij}=A_{ij}B_{ij}$ for $1\le i, j \le n$. We refer to $Q_n$ as the support matrix of $M_n$ as in \cite{cookregular}. Conditioning on $Q_n$, the nonzero entries of $M_n$ are independent signs. The support matrix $Q_n$ can be viewed as the adjacency matrix of a directed graph (digraph) on $n$ labeled vertices in which every vertex has an out-degree (row sum) of $d$, while the in-degrees (column sums) are random with a mean of $d$ (allowing self-loops).
If one allows the degree $d=n$, there are no zero entries, and the model reduces to the classical i.i.d.\ sign matrix. Hence, signed random combinatorial matrices connect the classical i.i.d.\ sign matrix with (unsigned) random combinatorial matrices. 

There are two closely related models that provide useful points of comparison. The first is Cook's signed regular digraph model \cite{cookregular}, in which both the row and column supports are constrained. More precisely, let $d\le n$ be positive integers and let $A^\pm=(A^\pm_{ij})_{1\le i, j\le n}$ be uniformly drawn from the set of matrices in $\{-1, 0, 1\}^{n\times n}$ with exactly $d$ nonzero entries in each row and in each column. Then $A^\pm\stackrel{d}{=} B\circ  \Xi$, where $B$ is the adjacency matrix of a uniform $d$ regular digraph and $\Xi$ is an independent sign matrix. Cook proved that this signed model is singular with probability $O(d^{-1/4})$ for $C(\log n)^2 \le d\le n$, and more generally established invertibility results for matrices of the form $\Sigma\circ  \Xi$ under expansion assumptions on the support matrix $\Sigma$ \cite[Theorem~1.3 and Theorem~1.4]{cookregular}. This work motivates the support and sign representation used above. Our model is the row regular analogue of this signed model without column sum constraint.

The second related model is the (unsigned) random combinatorial matrix $Q_n$ defined as above. Nguyen initiated the study of singularity probability for this model in \cite{nguyencombinatorial}.  For $d=n/2$, Tran  \cite{tran} obtained the sharp lower bound estimate for the smallest singular value, and Jain, Sah, and Sawhney \cite[Theorem 1.8]{jaindiscrete} confirmed Nguyen's conjecture on the sharp singularity probability in \cite[Conjecture 1.4]{nguyencombinatorial}. More recently, in the dense regime $d=pn$ with a fixed constant $p\in(0,1/2]$, D. Li, A. Litvak, and T. Yu \cite{lianupper} proved the complementary upper bound
\[
\mathbb{P}\left( s_n(Q_n) \le \frac{\sqrt{d}}{\varepsilon^2 n} \right)
\ge 1-C_p\left(\varepsilon + \frac{1}{\sqrt{d}}\right), \qquad \forall \varepsilon>0,
\]
where $C_p>0$ depends only on $p$. Together with the corresponding lower bound estimates \cite{tran}, these results show that the smallest singular value of $Q_n$ is typically of order $n^{-1/2}$ in the dense regime. However, the sparse regime, for $C\log n\le d \le n/2$ and $d=o(n)$, remains less understood. Related singularity results for sparse random combinatorial matrices were obtained in \cite{aignerperson,ferberkwansauermann}. D. Li, A. Litvak, and T. Yu recently obtained the first quantitative polynomial lower bound on the smallest singular value of shifted sparse random combinatorial matrices \cite[Theorem 1.7]{lianlitvakyucircular}.

In the sequel, we refer to $Q_n$ as the unsigned model and to $M_n$ as the signed model.
In the case of $d=n/2$, K. Yu \cite{yusigned} recently showed that there exist absolute constants $C, c>0$ such that for any $\varepsilon\ge 0$,
\begin{equation}\label{eq:1}
\mathbb{P}\left(s_n(M_n)\le \varepsilon n^{-1/2}\right)\le C\varepsilon +e^{-cn}.
\end{equation}
In particular, $M_n$ is singular with exponentially small probability. Moreover, the above results remain valid for $d=pn$ with a fixed constant $p\in(0,1/2]$  (see \cite[Remark 1.2]{yusigned}). 
This result suggests that \(n^{-1/2}\) is the correct scale for the smallest singular value in the dense regime, but until recently, a matching upper bound was not known.
In this paper, our main contribution is to establish a corresponding upper bound for $s_n(M_n)$. We obtain the following main result.
\begin{theorem}\label{theorem 1.1}
Let $p\in (0,1/2]$ be a fixed constant, and let $d, n$ be positive integers such that $d=pn$. Let $M_n$ be an $n\times n$ random matrix, whose rows are independent vectors uniformly distributed over the set of all vectors in $\{-1,0,1\}^{n}$ with exactly $d$ nonzero entries. Then there exists a constant  $C_p>0$ depending only on $p$ such that for every $\varepsilon> 0$,
\begin{equation}
\mathbb{P}\left( s_n(M_n) > \frac{\sqrt{d}}{\varepsilon^2 n} \right)
\le  C_p\left(\varepsilon + \frac{1}{\sqrt{d}}\right).
\nonumber
\end{equation}
\end{theorem}

\begin{remark}
Theorem \ref{theorem 1.1}, together with K. Yu's lower bound result \eqref{eq:1}, shows that in the dense regime $d=pn$ the smallest singular value $s_n(M_n)$ is typically of order $n^{-1/2}$ with probability tending to $1$, matching the behavior of $s_n(Q_n)$ of (unsigned) random combinatorial matrix. Moreover, the order $\sqrt{d}/n$ in Theorem \ref{theorem 1.1} is the one suggested by comparison with the classical i.i.d.\ model. Indeed, if $R$ is one row of $M_n$, then   
\[
    \mathbb E R= 0, \qquad \operatorname{Cov}(R)=\frac dn \mathbf{I}_n.
    \]
Thus, after rescaling by the standard deviation $\sqrt{d/n}$, the rows become isotropic. Comparison with a classical i.i.d.\ centered unit variance model suggests that the least singular value should have scale $n^{-1/2}$. Undoing this normalization suggests a scale of $\sqrt{d/n}\cdot n^{-1/2}=\sqrt{d}/n$ for $M_n$ (analogous scale predictions appear for $d$ regular digraph model \cite{litvakshiftedregular} and for the (unsigned) random combinatorial model \cite[Conjecture 1.5]{lianupper}).  
\end{remark}

\begin{remark}[Comparison with the unsigned model]\label{remark: comparison}
Compared with $Q_n$, the signs in $M_n$ center the row distribution and remove the deterministic constant direction.
Let $\xi$ be one row of the unsigned model $Q_n$ (uniformly distributed over $\{0,1\}^n$ with exactly $d$ ones), and write $p=d/n$. Then
    \[
    \mathbb E \xi =p \mathbf 1, \qquad \operatorname{Cov}(\xi)=\frac{p(1-p)n}{n-1}P_{\mathbf 1^\perp},
    \]
where $P_{\mathbf 1^\perp}=\mathbf I_n-\frac1n \mathbf 1\mathbf 1^T$ denotes the orthogonal projection onto $\mathbf 1^\perp:=\{v\in \mathbb{R}^n: \langle v, \mathbf 1\rangle=0\}$, and $\mathbf 1=(1,\dots, 1)$. The identity $\langle \xi, \mathbf 1\rangle =d$ creates a deterministic component in the $\mathbf 1$ direction. This is why arguments for $Q_n$ often separate the constant direction. This distinction also affects the net argument (see, for example, \cite{tran, lianupper, lianlitvakyucircular}). 
In contrast, the signed model is already centered and  does not require separating out the $\mathbf 1$ direction when estimating the operator norm and linear forms estimates. 

\end{remark}

\begin{remark}[Condition numbers]
    For the signed model, we have $\|M_n\|_{\mathrm{HS}}^2=nd$ deterministically. Since
    \[\|M_n\|_{\mathrm{HS}}^2=\sum_{i=1}^n s_i^2(M_n)\le n s_1^2(M_n),
    \]
  we have
    \[
    s_1(M_n)\ge \frac{\|M_n\|_{\mathrm{HS}}}{\sqrt{n}}=\sqrt{d}.
    \]
Combining with Lemma \ref{lemma 2.3}, $s_1(M_n)$ has order $\sqrt{n}$ for $d=pn$ with a fixed constant $p\in (0,1/2]$. Hence, the condition number of the signed model has typical order 
\[\kappa(M_n)=\frac{s_1(M_n)}{s_n(M_n)} \approx n
\] 
in the dense regime (i.e., for every $\eta>0$, there are constants $C, c>0$ depending only on $p$ and $\eta$, such that $\mathbb P(cn\le \kappa (M_n)\le Cn)\ge 1-\eta$ for all sufficiently large $n$). This has important implications for numerical computation. Indeed, it is well known that condition number is closely linked to computational complexity, particularly in the context of numerical linear algebra and iterative methods. In general, a large condition number typically leads to slower convergence \cite{allaire2008numerical}. This contrasts with the unsigned model.  Since $Q_n \mathbf 1=d \mathbf 1$, one has $s_1(Q_n)\ge d$, and centered operator norm estimates imply $s_1(Q_n)=O(d)$ \cite[Theorem 4.1]{lianlitvakyucircular}. Together with $s_n(Q_n)\approx n^{-1/2}$, this yields condition number $\kappa(Q_n)$ of order $n^{3/2}$ in the dense regime (see \cite[Remark 1.3]{lianupper}).
\end{remark}

\subsection{Organization}
The remainder of the paper is organized as follows. In Section \ref{sec2}, we present some preliminaries. In Section \ref{sec3}, we establish the invertibility estimates for fixed vectors and over the set of almost-constant vectors. In Section \ref{sec4}, we prove Theorem \ref{theorem 1.1}.   

\section{Preliminaries}\label{sec2}

\subsection{Notation}
We adopt the following standard notations. We write $[n]:=\{1,2,\dots, n\}$ and denote by $\{e_i\}_{
i=1}^n$ the standard basis of $\R^n$. The standard inner product on $\R^n$ is denoted by $\ip{\cdot}{\cdot}$, and the Euclidean norm by $\norm{\cdot}_2$. The vector $\mathbf 1\in \mathbb R^n$ denotes the vector with each coordinate equal to $1$, and $\mathbf{I}_n$ denotes the $n\times n$ identity matrix. We denote the unit sphere in $\R^n$ by $\bS^{n-1}:=\{x\in \mathbb R^n: \|x\|_2=1\}$ and the Euclidean unit ball by $B_2^n$. The distance from a vector $v$ to a subspace $H\subset \mathbb R^n$ is defined as \[\operatorname{dist}(v,H):=\inf_{u\in H}\norm{v-u}_2.\] For an $n\times n$ matrix $A=(a_{ij})$, its operator norm is denoted by $\norm{A}:=\sup_{x\in \mathbb S^{n-1}}\|Ax\|_2$ and its Hilbert-Schmidt norm by $\norm{A}_{\mathrm{HS}}:=\sqrt{\sum_{1\le i, j\le n}a_{ij}^2}$. The transpose and inverse of $A$ are denoted by $A^T$ and $A^{-1}$, respectively. For $n\times n$ matrices $A$ and $B$, $A\circ B$ denotes the Hadamard (entrywise) product (i.e., $(A\circ B)_{ij}=A_{ij}B_{ij}$ for $1\le i, j\le n)$. For two random variables $X$ and $Y$, we write $X\stackrel{d}{=} Y$ to denote equality in distribution. We write $\mathbf 1^\perp:=\{v\in \mathbb{R}^n: \langle v, \mathbf 1\rangle=0\}$ and the orthogonal projection onto $\mathbf 1^\perp$ is denoted by $P_{\mathbf 1^\perp}=\mathbf I_n-\frac1n \mathbf 1\mathbf 1^T$. For a random vector $X$, its covariance matrix is denoted by $\operatorname{Cov}(X)$. Finally, $C_0,c_0,C_1,c_1,\dots$ denote positive constants that may depend on the parameter $p=d/n$.

\subsection{Almost-constant vectors}

An often used approach to estimating the smallest singular value, introduced in \cite{litvak2005smallest} (see also \cite{litvak2004random}), is to partition the unit sphere and treat vectors of different complexity separately. In this context, one usually separates vectors of small complexity, close to sparse vectors, from spread vectors. This splitting was later formalized into the compressible/incompressible method by Rudelson and Vershynin \cite{thelittle}. For non-Hermitian models with combinatorial dependence, such as random regular digraphs and random combinatorial matrices, it is more natural to use a related decomposition into almost-constant vectors and its complement. This dichotomy goes back to Cook's work on random regular digraphs \cite{cookregular} and was further used and refined in \cite{litvakdigraphs,litvakshiftedregular,tran,jaindiscrete,lianupper,lianlitvakyucircular,yusigned}. We shall use this latter decomposition below. The almost-constant part has low metric complexity and will be handled by a net argument, while the complement is the regime where CLCD-based anti-concentration estimates apply.

\begin{definition}\label{definition 2.1}
Let $\delta, \rho \in (0,1)$. We define $\operatorname{Cons}(\delta, \rho)$ as the set of vectors $v \in \bS^{n-1}$ such that $|v_i - \lambda| \le \rho/\sqrt{n}$ holds for some $\lambda \in \R$ and for at least $(1-\delta)n$ coordinates $i \in [n]$. A vector $v \in \bS^{n-1}$ is called non-almost-constant if  $v \in \bS^{n-1}\setminus \operatorname{Cons}(\delta, \rho)$.
\end{definition}

\subsection{Biorthogonal systems}

We next recall the notion of a biorthogonal system, which plays the role of a dual basis and will be used to relate inverse matrices to distances from vectors to the spans of the remaining vectors.
A system $\{E_i,F_i\}_{i=1}^n$ in a Hilbert space $\mathcal{H}$ is called a \textit{biorthogonal system} if $\ip{E_i}{F_j}=\delta_{ij}$, for all $i,j\in[n]$. Note that in a biorthogonal system, both $\{E_i\}_{i=1}^n$ and $\{F_i\}_{i=1}^n$ are linearly independent. Consequently, if $\mathcal{H}$ has dimension $n$, then each set spans $\mathcal{H}$, and the system is said to be \textit{complete}. The following proposition collects several properties of biorthogonal systems that are useful for the proof of Theorem \ref{theorem 1.1} (see \cite[Proposition 2.1]{theleast}).
\begin{proposition}\label{proposition 4.1}
\begin{enumerate}
\item Let $\{E_i\}_{i=1}^n$ be linearly independent vectors in an $n$-dimensional Hilbert space $\mathcal{H}$. There exist unique vectors $\{F_i\}_{i=1}^n$ such that $\{E_i, F_i\}_{i=1}^n$ forms a complete biorthogonal system in $\mathcal{H}$.
\item If $\{E_i, F_i\}_{i=1}^n$ is a complete biorthogonal system in $\mathcal{H}$, then the Euclidean norm of each $F_i$ satisfies
\[
\norm{F_i}_2=\frac{1}{\operatorname{dist}(E_i,H_i)},\quad  i\in[n],
\]
where $H_i=\mathrm{span}\{E_j\}_{j\neq i}$.
\item If A is an $n\times n$ invertible matrix, then $\left\{Ae_i, (A^{-1})^Te_i\right\}_{i=1}^n$ is a complete biorthogonal system in $\R^n$.
\end{enumerate}
\end{proposition}

\subsection{Anti-concentration}


Anti-concentration estimates bound the probability that a random variable falls in a short interval. They are opposite to concentration inequalities and are a key tool in smallest singular value estimates, where one needs anti-concentration estimates for linear forms of the rows of a random matrix.
To measure anti-concentration, we use the Lévy concentration function.
\begin{definition}\label{definition 3.1}
Let $\xi$ be a random variable. The \textit{Lévy concentration function} of $\xi$ is defined by
\[
\mathcal{L}(\xi,\varepsilon):=\sup_{x \in \R}\mathbb{P}\{|\xi-x| \le \varepsilon\}, \quad \varepsilon \ge 0.
\]
\end{definition}
A standard tool for bounding the Lévy concentration function of sums of i.i.d.\ random variables is the Least Common Denominator (LCD), introduced by Rudelson and Vershynin in \cite{thelittle}. However, in our setting, the entries within each row of $M_n$ are not independent, since each row contains exactly $d$ nonzero entries, each taking value either $-1$ or $1$. To handle this combinatorial dependence, Tran \cite{tran} introduced the Combinatorial Least Common Denominator (CLCD), which is specifically designed for $0/1$ (unsigned) random combinatorial matrix. K. Yu \cite{yusigned} then developed the corresponding signed anti-concentration estimates for the $\{-1, 0, 1\}$ signed model. We use these CLCD anti-concentration estimates below.
\begin{definition}\label{definition 3.2}
For a vector $v\in \R^n$, and parameters $\alpha,\gamma>0$, define 
\[
\operatorname{CLCD}_{\alpha,\gamma}(v):=\inf\left\{\theta>0:\operatorname{dist}\left(\theta \mathrm{D}(v),\mathbb{Z}^{\binom{n}{2}}\right)<\min\left(\gamma\norm{\theta\mathrm{D}(v)}_2,\alpha\right)\right\}.
\]
Here by $\mathrm{D}(v)$ we denote the vector in $\R^{\binom{n}{2}}$ whose $(i,j)$-coordinate is $v_i-v_j$, for $1\le i<j\le n$.
\end{definition}
The following lemma provides a lower bound on the CLCD for non-almost-constant vectors (see \cite[Lemma 2.15]{tran}).
\begin{lemma}\label{lemma 3.3}
Let $\delta, \rho\in(0,1)$, and fix $v\in \bS^{n-1}\setminus \operatorname{Cons}(\delta, \rho)$. Then, for every $\alpha>0$ and every $\gamma \in (0,\delta\rho/12)$, we have
\[
\operatorname{CLCD}_{\alpha,\gamma}(v)\ge\frac{\sqrt{\delta n}}{7}.
\]
\end{lemma}

For $p=1/2$, K. Yu \cite[Theorem 1.3]{yusigned} obtained anti-concentration estimates for the random sum $\sum_{i=1}^n v_i \eta_i$. The same proof applies to every fixed $p\in (0,1/2]$. We omit the proof.
\begin{lemma}\label{lemma 3.4}
Fix any $b>0$ and $\gamma \in(0,1)$. Let $d=pn$ be a positive integer with $p\in(0,1/2]$. Let $v\in \R^n $ such that $\norm{\mathrm{D}(v)}_2\ge b\sqrt{n}$. Suppose $\eta$ is a random vector uniformly distributed over the set of vectors in $\{-1,0,1\}^{n}$ with exactly $d$ nonzero entries. Then, for every $\alpha>0$ and $\varepsilon\ge0$, we have
\[
\mathcal{L}\left(\sum_{i=1}^nv_i\eta_i,\varepsilon\sqrt{p}\right)\le C\varepsilon +\frac{C}{\operatorname{CLCD}_{\alpha,\gamma}(v)}+Ce^{-4\alpha^2d/{n}^2}.
\]
The constant $C>0$ depends only on $p, b$, and $\gamma$.
\end{lemma}

\begin{remark}\label{remark 3.7}
For any $v\in\bS^{n-1}\setminus \operatorname{Cons}(\delta, \rho)$, it follows from \cite[Lemma 2.2]{tran} that $\norm{\mathrm{D}(v)}_2\ge \delta\rho\sqrt{n}/8\sqrt{2}$. In Lemma \ref{lemma 3.4}, we set $b=\delta\rho/8\sqrt{2}$ and $\alpha=\mu n$ with $\mu>0$, and let $0<\gamma<\delta\rho/12$.
Combining Lemma \ref{lemma 3.3}, we then obtain
\[
\mathcal{L}\left(\sum_{i=1}^nv_i\eta_i,\varepsilon\sqrt{p}\right)\le C\varepsilon +\frac{C}{\sqrt{d}}+Ce^{-4\mu^2d},
\]
where the constant $C>0$ depends only on $p, \delta,\rho$, and $\gamma$.
\end{remark}

\subsection{Operator norm}

The following lemma establishes an upper bound on the operator norm of $M_n$. Specifically, for $d=n/2$, K. Yu \cite[Proposition 2.1]{yusigned} showed that, with high probability, $\norm{M_n}=O(\sqrt{n})$. This result is generalized below to the regime $d=pn$ for $p\in(0,1/2]$. The proof is analogous to that of \cite[Proposition 2.1]{yusigned} and is therefore omitted.
\begin{lemma}\label{lemma 2.3}
Let $p\in(0,1/2]$ and $d\le n$ be positive integers such that $ d=pn$. Let $M$ be an $m\times n$ random matrix, $1\le m \le n$, whose rows are independent vectors uniformly distributed over the set of vectors in $\{-1,0,1\}^{n}$ with exactly $d$ nonzero entries. Then there exist positive constants $C_{\ref{lemma 2.3}},c_{\ref{lemma 2.3}}$ depending only on p such that for all $t\ge C_{\ref{lemma 2.3}}$, we have
\[
\mathbb{P}\left(\norm{M}\ge t\sqrt{n}\right)\le2e^{-c_{\ref{lemma 2.3}}t^2n}.
\]
\end{lemma}

\section{Fixed vectors and almost-constant vectors estimates}\label{sec3}

\subsection{Invertibility over fixed vectors}

We need the following lemma to establish the invertibility of $M_n$ on a single vector $v \in \bS^{n-1}$.
\begin{lemma}\label{lemma 2.2}
Let $p\in(0,1/2]$ and $d\le n$ be positive integers such that $ d=pn$. There exist positive constants $C_{\ref{lemma 2.2}},c_{\ref{lemma 2.2}}$ depending only on $p$ such that the following holds. Let $M$ be an $m\times n$ random matrix, $n/2 \le m\le n$, 
whose rows are independent vectors uniformly distributed over the set of all vectors in $\{-1,0,1\}^{n}$ with exactly $d$ nonzero entries. Then for any fixed vector $v\in \bS^{n-1} $, we have
\begin{equation}
\mathbb{P} \left(\norm{Mv}_2\le C_{\ref{lemma 2.2}}\sqrt{n}\right)\le e^{-c_{\ref{lemma 2.2}}n}.
\nonumber
\end{equation}
\end{lemma}
\begin{proof}
Let $q$ be a random vector uniformly distributed over the set of all vectors in $\{-1,0,1\}^n$ with exactly $d$ nonzero entries. By definition, $q$ and $\xi \circ \eta$ are identical in distribution, where $\xi$ is uniformly distributed over the set of all vectors in $\{0,1\}^{n}$ with exactly $d$ ones and $\eta$ is an i.i.d.\ sign vector with $\mathbb{P} \left( \eta_i=1\right)=\mathbb{P}\left(\eta_i=-1\right)=1/2$ for $i\in[n]$, where $\xi$ and $\eta$ are independent vectors. Then one can write
\[
q_i \stackrel{d}{=} \eta_i \xi_i,  \qquad \forall \, i \in [n].
\] 
Since $\eta_i$ is independent of $\xi_i$ and satisfies $\mathbb{E}(\eta_i)=0$, we have 
\begin{equation}\label{eq:2}
\mathbb{E} (q_i)=\mathbb{E}\left(\eta_i \xi_i\right)=\mathbb{E}(\eta_i)\mathbb{E}(\xi_i)=0. 
\end{equation}
If $i\ne j$, we have $\mathbb{E}(\eta_i\eta_j)=\mathbb{E}(\eta_i)\mathbb{E}(\eta_j)=0$. If $i=j$, we have $\mathbb{E}(\eta_i^2)=1$ and $\mathbb{E}(\xi_i^2)=d/n$. Then we obtain
\begin{equation}
\label{eq:3}
\operatorname{Cov}\left(q_i,q_j\right)=\mathbb{E}(q_i q_j)=\mathbb{E}\left(\eta_i \xi_i\eta_j \xi_j\right)=\mathbb{E}\left(\eta_i \eta_j\right)\mathbb{E}\left(\xi_i\xi_j\right)=\frac{d}{n}\delta_{ij}.
\end{equation}
For any fixed $v\in \bS^{n-1}$, we
denote $X:=\sum_{i=1}^n v_iq_i$. Since (\ref{eq:2}) and (\ref{eq:3}), we obtain 
\[
\mathbb{E} X=\mathbb{E}\left(\sum_{i=1}^{n}v_iq_i\right)=\sum_{i=1}^nv_i\mathbb{E}(q_i)=0,
\] 
and
\[
\mathbb{E} X^2=\mathbb{E}\left[\left(\sum_{i=1}^{n}v_iq_i\right)^2\right]=\sum_{i=1}^n v_i^2\mathbb{E}\left(q_i^2\right)+\sum_{i\ne j}v_i v_j\mathbb{E}(q_i q_j)=\frac{d}{n}=p.
\]

Next, we show that $X$ is a subgaussian random variable. For any $\lambda \in \R$, we have
\[
\mathbb{E} e^{\lambda X}=\mathbb{E}_\xi\left[\mathbb{E}_\eta\exp\left({\lambda\sum_{i=1}^nv_i\eta_i\xi_i}\right)\right],
\]
where $\mathbb{E}_\eta$ and $\mathbb{E}_\xi$ denote the expectations taken with respect to $\eta$ and $\xi$. For each fixed $\xi
$, the random variables $\eta_1,\dots,\eta_n$ are independent, so
\[
\mathbb{E}_\eta\exp\left({\lambda\sum_{i=1}^nv_i\eta_i\xi_i}\right)=\prod_{i=1}^n\mathbb{E}_\eta \exp\left(\lambda v_i \eta_i \xi_i\right).
\]
Since $\eta_i$ is a Rademacher random variable, we have
\[
\mathbb{E}_\eta \exp\left(\lambda v_i \eta_i \xi_i\right)=\cosh(\lambda v_i\xi_i)\le\exp\left(\frac{\lambda^2v_i^2\xi_i^2}{2}\right)\le\exp\left(\frac{\lambda^2v_i^2}{2}\right),
\]
where the first inequality follows from the standard bound $\cosh{t}\le e^{t^2/2}$ for all $t\in \mathbb{R}$ and the last inequality follows from $\xi_i^2\le 1$.

Finally, we obtain
\begin{align*}
\mathbb{E} e^{\lambda X}&=\mathbb{E}_\xi\left[\mathbb{E}_\eta\exp\left({\lambda\sum_{i=1}^nv_i\eta_i\xi_i}\right)\right]\\
&=\mathbb{E}_\xi\left[\prod_{i=1}^n\mathbb{E}_\eta \exp\left(\lambda v_i \eta_i \xi_i\right)\right]\\
&\le\mathbb{E}_\xi \left[\prod_{i=1}^n\exp\left(\frac{\lambda^2v_i^2\xi_i^2}{2}\right)\right]\\
&\le\exp\left(\frac{\lambda^2}{2}\sum_{i=1}^n v_i^2\right)=\exp\left(\frac{\lambda^2}{2}\right).
\end{align*}
Thus, $X$ is a mean-zero subgaussian random variable. Then there exists an absolute constant $c_0$ such that the following holds 
\[
\mathbb{P} \left(|X| \ge t\right)\le2\exp\left(-c_0t^2\right),
\]for all $t\ge 0$. Hence, we have the following estimates
\[
\mathbb{E} X^4=\int_0^\infty4t^3\mathbb{P}\left(|X|\ge t\right)dt\le\int_0^\infty8t^3\exp\left(-c_0t^2\right)dt=\frac{4}{c_0^2}.
\]
We denote $c_1:=\max\{4/c_0^2,1\}$. By the Paley-Zygmund inequality and $\mathbb{E} X^2=p\in(0,1/2]$, we have
\[
\mathbb{P}\left( X^2\ge\frac{1}{2}p(1-p)\right)\ge\mathbb{P}\left(X^2\ge\frac{1}{2}\mathbb{E} X^2\right)\ge\frac{(\mathbb{E} X^2)^2}{4\mathbb{E} X^4}\ge\frac{p^2}{4c_1}.
\]
Since $1-p\ge1/2$, we have
\begin{equation}\label{eq:4}
\mathbb{P}\left(|X|\ge\frac{\sqrt{p}}{2}\right)\ge\mathbb{P}\left( X^2\ge\frac{1}{2}p(1-p)\right)\ge\frac{p^2}{4c_1}=:c_2\in(0,1/2],
\end{equation}
where $c_2$ depends only on $p$.

Let $q_1,\dots,q_m$ be the independent rows of $M$ and denote $X_i:=\ip{q_i}{v}$, $Y_i:=\mathbf{1}_{\{|X_i|\ge\sqrt{p}/2\}}$, for $i \in[m]$. From (\ref{eq:4}), we have $\mathbb{P}(Y_i=1)\ge c_2$. Let $G:=\sum_{i=1}^mY_i$. We have 
\[
\mathbb{E} G=\sum_{i=1}^m \mathbb{E} Y_i\ge c_2m\ge\frac{c_2n}{2}.
\]
Since $\{Y_i\}_{i=1}^{m}$ are independent Bernoulli random variables, we use the Chernoff inequality (see, e.g., \cite[Remark 2.3.2]{Vershynin_2026} lower tails with $t=\mu/2$), obtaining
\[
\mathbb{P}\left(G\le\frac{1}{2}\mathbb{E} G\right)\le \exp\left( -\frac{1}{8}\mathbb{E} G\right).
\]
Since $\mathbb{E} G\ge c_2m\ge c_2n/2$, we have
\[
\mathbb{P}\left(G\le\frac{c_2n}{4}\right)\le\mathbb{P}\left(G\le\frac{1}{2}\mathbb{E} G\right)\le\exp\left(-\frac{1}{8} \mathbb{E} G\right)\le\exp\left(-\frac{c_2n}{16}\right).
\]
Consider the $\norm{Mv}_2^2=\sum_{i=1}^{m}X_i^2$. If the event $\{G>c_2n/4\}$ occurs , then we obtain
\[
\norm{Mv}_2^2=\sum_{i=1}^{m}X_i^2\ge\sum_{i\in[m]:Y_i=1}X_i^2 \ge\frac{c_2n}{4}\left(\frac{\sqrt{p}}{2}\right)^2=\frac{c_2np}{16}.
\]
Therefore,
\[
\mathbb{P}\left(\norm{Mv}_2>\frac{\sqrt{c_2p}}{4}\sqrt{n}\right)\ge1-\exp\left(-\frac{c_2n}{16}\right),
\] let $C_{\ref{lemma 2.2}}=\sqrt{c_2p}/4$ and $c_{\ref{lemma 2.2}}=c_2/16$. This completes the proof.
\end{proof}

\subsection{Invertibility over almost-constant vectors}

The main goal of this section is to establish an invertibility estimate for $M_n$ on the set of almost-constant vectors. More precisely, we prove the following proposition, which shows that $\norm{M_nv}_2$ is uniformly bounded away from zero for all $v\in\operatorname{Cons}(\delta,\rho)$.
\begin{proposition}\label{proposition 2.4}
Let $p\in(0,1/2]$ and $d,n$ be sufficiently large positive integers such that $d=pn$. There exist constants $C_{\ref{proposition 2.4}},c_{\ref{proposition 2.4}},\delta, \rho\in(0,1)$ depending only on $p$ such that the following holds. Let $M$ be an $m \times n$ random matrix, $n/2 \le m\le n$, whose rows are independent vectors uniformly distributed over the set of all vectors in $\{-1,0,1\}^{n}$ with exactly $d$ nonzero entries. Then
\[
\mathbb P\left(\inf_{v\in \operatorname{Cons}(\delta,\rho)}\norm{M v}_2\le C_{\ref{proposition 2.4}}\sqrt{n}\right)\le e^{-c_{\ref{proposition 2.4}}n}.
\]
\end{proposition}
The proof of Proposition \ref{proposition 2.4} relies on a discretization argument. Specifically, we use the following net for the set of almost-constant vectors; its detailed construction can be found in the proof of \cite[Lemma 2.12]{tran}.
\begin{lemma}\label{lemma 2.5}
Let $\delta, \rho\in(0,1/12)$, and suppose that $n$ is sufficiently large with respect to $\delta$ and $\rho$. Then there exists a $(\delta+2\rho)$-net $\mathcal{N}\subset\frac{3}{2}B^n_2\setminus\frac{1}{2}B^n_2$ for $\operatorname{Cons}(\delta,\rho)$ with respect to the Euclidean norm, such that $|\mathcal{N}|\le\exp(2\delta\log(5/\delta)n)$. 
\end{lemma}
\begin{remark}
In the proof below, the approximation error is bounded by the operator norm of $M_n$. Indeed,  if $v\in \operatorname{Cons}(\delta,\rho)$ and $u\in \mathcal{N}$ such that $\|v-u\|_2\le \delta+2\rho$, then on the event $\|M_n\|\le C_{\ref{lemma 2.2}}\sqrt{n}$,
\[
\norm{M_n(v-u)}_2\le (\delta+2\rho) C_{\ref{lemma 2.2}}\sqrt{n}.
\]
Thus, choosing $\delta, \rho$ sufficiently small gives the required net approximation error (see Remark~\ref{remark: comparison} for a comparison with the unsigned model). 
\end{remark}

\begin{proof}
Let $C_{\ref{lemma 2.2}},c_{\ref{lemma 2.2}},C_{\ref{lemma 2.3}},c_{\ref{lemma 2.3}}$ be the constants from Lemma \ref{lemma 2.2} and Lemma \ref{lemma 2.3}. We assume $\delta ,\rho\in(0,1/12)$ is sufficiently small and will choose it later to satisfy the conditions. We will estimate the probability of the following event
\[
\mathcal{E}:=\left\{ \inf_{v\in\operatorname{Cons}(\delta,\rho)}\norm{Mv}_2\le\frac{C_{\ref{lemma 2.2}}}{4}\sqrt{n}\ \text{and}\ \norm{M}< C_{\ref{lemma 2.3}}\sqrt{n}\right\}.
\]   
By Lemma \ref{lemma 2.3}, we have
\[       \mathbb{P}\left(\norm{M}\ge C_{\ref{lemma 2.3}}\sqrt{n}\right)\le 2e^{-cn},
\]
where $c=c_{\ref{lemma 2.3}}C_{\ref{lemma 2.3}}^2$.
Next, we suppose $\mathcal{E}$ occurs. Then 
\[
\norm{M}<C_{\ref{lemma 2.3}}\sqrt{n},
\]
and there exists a vector $v\in\operatorname{Cons}(\delta,\rho)$ such that
\[
\norm{Mv}_2\le\frac{C_{\ref{lemma 2.2}}}{4}\sqrt{n}.
\]
By Lemma \ref{lemma 2.5}, there is a vector $u\in\mathcal{N}$ such that 
\[1/2\le\norm{u}_2\le3/2,\qquad\norm{v-u}_2\le\delta+2\rho.
\]
Thus, we obtain
\[
\norm{Mu}_2\le\norm{Mv}_2+\norm{M(v-u)}_2\le\frac{C_{\ref{lemma 2.2}}}{4}\sqrt{n}+(\delta+2\rho)C_{\ref{lemma 2.3}}\sqrt{n}\le\frac{C_{\ref{lemma 2.2}}}{2}\sqrt{n},
\]
provided that
\begin{equation}\label{eq:5}
\delta+2\rho\le\frac{C_{\ref{lemma 2.2}}}{4C_{\ref{lemma 2.3}}},
\end{equation}
and $n$ is large enough. Then we have
\[
\mathbb{P}(\mathcal{E})\le\mathbb{P}\left(\inf_{u\in\mathcal{N}}\norm{Mu}_2\le\frac{C_{\ref{lemma 2.2}}}{2}\sqrt{n}\right)\le|\mathcal{N}|\max_{u\in\mathcal{N}}\mathbb{P}\left(\norm{Mu}_2\le\frac{C_{\ref{lemma 2.2}}}{2}\sqrt{n}\right).
\]
By Lemma \ref{lemma 2.2}, we obtain
\begin{align*}
\mathbb{P}\left(\norm{Mu}_2\le\frac{C_{\ref{lemma 2.2}}}{2}\sqrt{n}\right)=\mathbb{P}\left(\left\|M\frac{u}{\norm{u}_2}\right\|_2\le\frac{C_{\ref{lemma 2.2}}}{2\norm{u}_2}\sqrt{n}\right)&\le\mathbb{P}\left(\norm{Mu'}_2\le C_{\ref{lemma 2.2}}\sqrt{n}\right)\le e^{-c_{\ref{lemma 2.2}}n}.
\end{align*}

Combining the above estimates, we obtain 
\[
\mathbb{P}(\mathcal{E})\le|\mathcal{N}|\max_{u\in\mathcal{N}}\mathbb{P}\left(\norm{Mu}_2\le\frac{C_{\ref{lemma 2.2}}}{2}\sqrt{n}\right)\le\exp(2\delta\log(5/\delta)n)\cdot e^{-c_{\ref{lemma 2.2}}n}\le\exp\left(-\frac{c_{\ref{lemma 2.2}}}{2}n\right),
\]assuming that $\delta$ satisfies
\begin{equation}\label{eq:6}
2\delta\log(5/\delta)\le\frac{c_{\ref{lemma 2.2}}}{2}.
\end{equation}
We choose $\delta$ and $\rho$ sufficiently small so that conditions (\ref{eq:5}) and (\ref{eq:6}) are satisfied. Assuming $n$ is sufficiently large, we then obtain
\begin{align*}
\mathbb{P}\left(\inf_{v\in \operatorname{Cons}(\delta,\rho)}\norm{Mv}_2\le C_{\ref{proposition 2.4}}\sqrt{n}\right)&\le \mathbb{P}(\mathcal{E})+\mathbb{P}\left(\norm{M}\ge C_{\ref{lemma 2.3}}\sqrt{n}\right)\le
e^{-c_{\ref{proposition 2.4}}n},
\end{align*}
where $C_{\ref{proposition 2.4}}=C_{\ref{lemma 2.2}}/4$, and the constant $c_{\ref{proposition 2.4}}$ is chosen appropriately. This completes the proof.
\end{proof}

\section{Proof of Theorem \ref{theorem 1.1}}\label{sec4}

The proof follows the framework introduced by Rudelson and Vershynin in \cite{theleast} (see also \cite{tao2023topics, anuppersquare, lianupper}).
The key observation is the following: when $M_n$ is invertible,
\[
\exists x \in \R^n: \norm{x}_2\le u,\quad \norm{(M_n^T)^{-1}x}_2\ge v\sqrt{n} \implies s_n(M_n)\le \frac{u}{v\sqrt{n}}.
\]
Thus, it suffices to construct such a vector $x$.

Let $X_i:=M_n^Te_i$ for all $i\in[n]$, and define $H_1:=\operatorname{span}\{X_2,\dots,X_n\}$. Denote by $P_1:\R^n\to H_1$ the orthogonal projection onto $H_1$, and let
\[
x:=X_1-P_1X_1.
\]
The following lemma provides a probabilistic upper bound on $\norm{x}_2$.
\begin{lemma}\label{lemma 4.2}
Let $d,n$ be positive integers with $d=pn$, where $p\in(0,1/2]$. Then for every $s>0$, the random vector $x=X_1-P_1X_1$ satisfies
\[
\mathbb{P}\left(\norm{x}_2\ge s\right)\le\frac{p}{s^2}+\mathbb{P}\left(\mathcal{E}^c_0\right)
\]
where $\mathcal{E}_0$ denotes the event that $\{M_n\ \text{is invertible}\}$.
\end{lemma}
\begin{remark}
Based on \cite[Remark 1.1 and Remark 1.2]{yusigned}, we obtain $\mathbb{P}\left(\mathcal{E}^c_0\right)\le e^{-cn}$, where $c>0$ depends only on $p$. Then we have
\begin{equation}\label{eq:7} \mathbb{P}\left(\norm{x}_2\ge s\right)\le\frac{p}{s^2}+e^{-cn}.
\end{equation}
\end{remark}
\begin{proof}
Note that
\[
\mathbb{P}\left(\norm{x}_2\ge s\right)\le\mathbb{P}\left(\norm{x}_2\ge s\ \text{and}\ \mathcal{E}_0\right)+\mathbb{P}\left(\mathcal{E}^c_0\right).
\]
Let $f_1$ be any unit vector orthogonal to $X_2,\dots,X_n$. 
We can choose $f_1$ as a random vector that depends only on $X_2,\dots,X_n$ and is independent of $X_1$. If $\mathcal{E}_0$ occurs, then $H_1$ has dimension $n-1$, and $H_1^{\perp}$ has dimension $1$. In this case, we have 
\[
\norm{x}_2=\operatorname{dist}(X_1,H_1)=|\ip{X_1}{f_1}|.
\]
Then we obtain
\[
\mathbb{P}\left(\norm{x}_2\ge s\ \text{and}\ \mathcal{E}_0\right)=\mathbb{P}\left(|\ip{X_1}{f_1}|\ge s\ \text{and}\ \mathcal{E}_0\right)\le\mathbb{P}\left(|\ip{X_1}{f_1}|\ge s\right)\le\frac{\mathbb{E}|\ip{X_1}{f_1}|^2}{s^2}.
\]
Denote the $i$-th coordinate of $X_1$ and $f_1$ by $X_1(i)$ and $f_1(i)$, respectively. Then 
\[
\mathbb{E}|\ip{X_1}{f_1}|^2=\mathbb{E}\left(\sum_{i=1}^n X_1(i)f_1(i)\right)^2=\sum_{i=1}^n\mathbb{E} f_1(i)^2X_1(i)^2+\sum_{i\neq j}\mathbb{E} f_1(i)f_1(j)X_1(i)X_1(j).
\]
By (\ref{eq:3}) and the independence of $f_1$ and $X_1$, we immediately obtain 
\[
\sum_{i=1}^n\mathbb{E} f_1(i)^2X_1(i)^2=\frac{d}{n}=p,\qquad
\sum_{i\neq j}\mathbb{E} f_1(i)f_1(j)X_1(i)X_1(j)=0.
\]
This establishes the desired result.
\end{proof}

Next, we consider the quantity $\norm{(M_n^T)^{-1}x}_2$, which is initially defined only on the event $\mathcal{E}_0$. We extend this quantity to the whole probability space by setting it equal to zero on $\mathcal{E}_0^c$. This extension preserves measurability and does not affect subsequent probabilistic estimates, since $\mathbb{P}(\mathcal{E}_0^c) \le e^{-cn}$.

On the event $\mathcal{E}_0$, \[\norm{(M_n^T)^{-1}x}_2=\norm{(M_n^T)^{-1}X_1-(M_n^T)^{-1}P_1X_1}_2=\norm{e_1-(M_n^T)^{-1}P_1X_1}_2.\] 
Since $P_1X_1\in H_1$, it follows that $(M_n^T)^{-1}P_1X_1\in\operatorname{span}\{e_2,\dots,e_n\}$. Consequently, $e_1$ is orthogonal to $(M_n^T)^{-1}P_1X_1$.
Thus, we obtain
\[
\norm{(M_n^T)^{-1}x}_2^2=\norm{e_1}_2^2+\norm{(M_n^T)^{-1}P_1X_1}_2^2\ge\norm{(M_n^T)^{-1}P_1X_1}_2^2.
\]
We denote $Y_i:=P_1M_n^{-1}e_i$. Since 
\[
\ip{X_1}{Y_1}=\ip{X_1}{P_1M_n^{-1}e_1}=\ip{P_1X_1}{M_n^{-1}e_1}=0,
\]
it follows that 
\begin{align*}
\norm{(M_n^T)^{-1}x}_2^2\ge\norm{(M_n^T)^{-1}P_1X_1}_2^2=\sum_{i=1}^n\ip{(M_n^T)^{-1}P_1X_1}{e_i}^2&=\sum_{i=1}^n\ip{X_1}{P_1M_n^{-1}e_i}^2=\sum_{i=2}^n\ip{X_1}{Y_i}^2.
\end{align*}
If $\mathcal{E}_0$ occurs, the following lemma from \cite[Lemma 2.1]{theleast} describes the relationship between the vectors $\{X_i\}_{i=2}^n$ and $\{Y_i\}_{i=2}^n$.
\begin{lemma}\label{lemma 4.4}
If $\{X_i\}_{i=2}^n$ and $\{Y_i\}_{i=2}^n$ are defined as above, then $\{X_i,\,Y_i\}_{i=2}^n$ forms a complete biorthogonal system in $H_1$.
\end{lemma}
Denote $H_{1,i}:=\operatorname{span}\{X_2,\dots,X_{i-1},X_{i+1},\dots,X_n\}, \, i\in\{2,\dots,n\}$. By Proposition \ref{proposition 4.1} and Lemma \ref{lemma 4.4}, we have $\norm{Y_i}_2=1/\operatorname{dist}(X_i,H_{1,i})$. Therefore, we have
\begin{equation}\label{eq:8}
    \norm{(M_n^T)^{-1}x}_2^2\ge\sum_{i=2}^n\ip{X_1}{Y_i}^2=\sum_{i=2}^n \left\langle X_1,  \frac{Y_i}{\norm{Y_i}_2}\right\rangle^2\cdot \norm{Y_i}_2^2=\sum_{i=2}^n\frac{a_i^2}{b_i^2},
\end{equation}
where we denote
\[
a_i:=\left|\left\langle X_1, {\frac{Y_i}{\norm{Y_i}_2}}\right\rangle\right|,\qquad b_i:=\norm{Y_i}_2^{-1}=\operatorname{dist}(X_i,H_{1,i})\quad \text{ for }\, i\in\{2,\dots,n\}.
\]
Note that $a_i$ and $b_i$ are well-defined only on $\mathcal{E}_0$. Equivalently, one can extend these definitions to the entire probability space by setting
\[
a_i := 
\begin{cases}
\left|\ip{X_1}{Y_i/\norm{Y_i}_2}\right|,&\text{on}\ \mathcal{E}_0 \qquad\\[1mm]
+\infty,&\text{on}\ \mathcal{E}_0^c
\end{cases}
\quad
b_i := 
\begin{cases}
\norm{Y_i}_2^{-1},& \text{on}\  \mathcal{E}_0\\[1mm]
0, & \text{on}\ \mathcal{E}_0^c
\end{cases}
\]
so that the quantities are well-defined and measurable on the entire space. This reduces the problem to establishing high-probability lower bounds for $a_i$ and upper bounds for $b_i$ for each $i\in\{2,\dots,n\}$. Without loss of generality, we focus on the case $i=2$, as the remaining cases are analogous.

First, we decompose the unit sphere $\bS^{n-1}$ into almost-constant vectors $\operatorname{Cons}(\delta,\rho)$ and non-almost-constant vectors $\bS^{n-1}\setminus \operatorname{Cons}(\delta, \rho)$. We will show that $H_{1,2}^\perp\cap \bS^{n-1}$ consists of non-almost-constant vectors with high probability. Consider the $(n-2)\times n$ matrix $A$, which is the submatrix of $M_n$ consisting of the rows $X_3,\dots,X_n$. Since $H_{1,2}$ is spanned by $\{X_3,\dots,X_n\}$, we have $H_{1,2}^\perp \subset \ker {A}$. By Proposition \ref{proposition 2.4}, there exist positive constants $C_{\ref{proposition 2.4}},c_{\ref{proposition 2.4}},\delta,\rho\in(0,1)$, depending only on $p$, such that
\begin{equation}\label{eq:9}
\mathbb{P}\left(H_{1,2}^\perp\cap\operatorname{Cons}(\delta,\rho)=\emptyset\right)\ge\mathbb{P}\left(\inf_{v\in \operatorname{Cons}(\delta,\rho)}\norm{Av}_2\ge C_{\ref{proposition 2.4}}\sqrt{n}\right)\ge 1-e^{-c_{\ref{proposition 2.4}}n}.
\end{equation}
Fix $\mu>0$ and $\gamma\in(0,\delta\rho/12)$. Denote the event
\[
\mathcal{E}_1:=\left\{\forall v\in H_{1,2}^\perp\cap\bS^{n-1}: \quad\operatorname{CLCD}_{\mu n,\gamma}(v)\ge\frac{\sqrt{\delta n}}{7}\quad   \mathrm{ and }\quad \norm{\mathrm{D}(v)}_2\ge \frac{\delta \rho\sqrt{n}}{8\sqrt{2}}\right\}.
\]
By (\ref{eq:9}), Lemma \ref{lemma 3.3}, and Remark \ref{remark 3.7} we have
\begin{equation*}
\mathbb{P}(\mathcal{E}_1)\ge\mathbb{P}\left(H_{1,2}^\perp\cap\operatorname{Cons}(\delta,\rho)=\emptyset\right)\ge 1-e^{-c_{\ref{proposition 2.4}}n}.
\end{equation*}
Denote by $P_{1,2}:\R^{n}\to H_{1,2}$ the orthogonal projection onto $H_{1,2}.$ Let 
\[
\Tilde{Z}:=X_2-P_{1,2}X_2.
\]
Define 
\[
Z:=
\begin{cases}
\Tilde{Z}/\norm{\Tilde{Z}}_2,& \quad  \Tilde{Z}\neq 0\\
0,&  \quad \Tilde{Z}=0
\end{cases}
\]
By construction, $Z\in H_1\cap H_{1,2}^{\perp}$. Since $P_{1,2}$ is a measurable function of $X_3,\dots,X_n$, it follows that $Z$ is measurable with respect to $\sigma(X_2,\dots,X_n)$. In particular, $Z$ is independent of $X_1$. On the event $\mathcal{E}_2:=\{\dim(H_1)=n-1\}$, we have $X_2\notin H_{1,2}$, and hence $\Tilde{Z}\neq0$. Therefore, on $\mathcal{E}_2$, $Z$ is a unit vector in $H_1\cap H_{1,2}^{\perp}$. Since this space is one-dimensional, on the event $\mathcal{E}_0\subset\mathcal{E}_2$, the vectors $Z$ and $Y_2/\norm{Y_2}_2$ differ only by a sign. Consequently, on $\mathcal{E}_0$, we obtain
\[
a_2=\left|\left\langle X_1, {\frac{Y_2}{\norm{Y_2}_2}}\right\rangle\right|=|\ip{X_1}{Z}|.
\]
Thus,
\begin{align*}
\mathbb{P}(a_2\le \varepsilon\sqrt{p})&\le\mathbb{P}(|\ip{X_1}{Z}|\le \varepsilon\sqrt{p})\\
&\le \mathbb{P}(\{|\ip{X_1}{Z}|\le \varepsilon\sqrt{p} \} \cap \mathcal{E}_1\cap  \mathcal{E}_2)+ \mathbb{P}(\mathcal{E}_1^c)+\mathbb{P}(\mathcal{E}_2^c)\\
&\le \mathbb{P}(\{|\ip{X_1}{Z}|\le \varepsilon\sqrt{p} \} \cap \mathcal{E}_1\cap  \mathcal{E}_2)+e^{-c_{\ref{proposition 2.4}}n}+e^{-cn},
\end{align*} where the last inequality follows from
\[
\mathbb{P}(\mathcal{E}_1^c)\le e^{-c_{\ref{proposition 2.4}}n},\qquad \mathbb{P}(\mathcal{E}_2^c)\le \mathbb{P}(\mathcal{E}_0^c)\le e^{-cn}.
\]
Conditioning on the event $\mathcal{E}_1\cap \mathcal{E}_2$, we have
\[
\norm{Z}_2=1,\qquad \operatorname{CLCD}_{\mu n,\gamma}(Z)\ge\frac{\sqrt{\delta n}}{7},\qquad   \norm{\mathrm{D}(Z)}_2\ge\frac{\delta \rho \sqrt{n}}{8\sqrt{2}}.
\]
By Lemma \ref{lemma 3.4} and Remark \ref{remark 3.7}, for every $\varepsilon \ge 0$, there exists a constant $C>0$ depending only on $p,\delta,\rho$, and $\gamma$ such that 
\[
\mathbb{P}\left(|\ip{X_1}{Z}|\le\varepsilon\sqrt{p}|X_2,\dots,X_n\right)\le C\varepsilon +\frac{C}{\sqrt{d}}+Ce^{-4\mu^2d}.
\]
Then
\begin{align*}
\mathbb{P}(\{|\ip{X_1}{Z}|\le\varepsilon\sqrt{p}\}\cap  \mathcal{E}_1\cap \mathcal{E}_2)&=\mathbb{E}_{X_2,\dots,X_n}\mathbb{P}(\{|\ip{X_1}{Z}|\le\varepsilon\sqrt{p}\}\cap  \mathcal{E}_1\cap \mathcal{E}_2|X_2,\dots,X_n)\\&\le C\varepsilon +\frac{C}{\sqrt{d}}+Ce^{-4\mu^2d}.
\end{align*}
Finally, we obtain the estimate
\begin{align*}
\mathbb{P}(a_2\le \varepsilon\sqrt{p})&\le\mathbb{P}(\{|\ip{X_1}{Z}|\le\varepsilon\sqrt{p}\}\cap  \mathcal{E}_1\cap \mathcal{E}_2)+e^{-c_{\ref{proposition 2.4}}n}+e^{-cn}\\
&\le C\varepsilon +\frac{C}{\sqrt{d}}+Ce^{-4\mu^2d}+e^{-c_{\ref{proposition 2.4}}n}+e^{-cn},
\end{align*}
where $C,c_{\ref{proposition 2.4}}$, and $c$ depend only on $p$.

To show that $b_2$ is bounded from above with high probability, we follow an argument analogous to the estimate (\ref{eq:7}). For every $t>0$
\[
\mathbb{P}(b_2\ge t)=\mathbb{P}(b_2\ge t\ \text{and}\ \mathcal{E}_0)=\mathbb{P}(\operatorname{dist}(X_2,H_{1,2})\ge t\ \text{and}\ \mathcal{E}_0),
\]
where the last equality follows from the fact that, if $\mathcal{E}_0$ occurs, $1/\norm{Y_2}_2=\operatorname{dist}(X_2,H_{1,2})$ Then we obtain
\[
\mathbb{P}(b_2\ge t)\le\frac{2p}{t^2}.
\]

Combining the probability estimates for $a_2$ and $b_2$, we obtain that for every $\varepsilon\ge0$ and $t>0$,
\[
\mathbb{P}(a_2\le\varepsilon\sqrt{p}\, \text{ or }\, b_2\ge t\sqrt{2p})\le C\varepsilon +\frac{C}{\sqrt{d}}+Ce^{-4\mu^2d}+e^{-c_{\ref{proposition 2.4}}n}+e^{-cn}+\frac{1}{t^2}.
\]
Repeating the same argument for all $i=3,\dots,n$, we conclude that, for every $t>0$ and $\varepsilon \ge 0$,
\begin{equation}\label{eq:10}
\mathbb{P}\left(\frac{a_i}{b_i}\le\frac{\varepsilon}{\sqrt{2}t}\right)\le C\varepsilon +\frac{C_1}{\sqrt{d}}+\frac{1}{t^2},
\end{equation}
for all sufficiently large $n$, where the constants $C$ and $C_1$ depend only on $p$.
We also require the following proposition to estimate the sum of $(a_i/b_i)^2$ in (\ref{eq:8}).
\begin{proposition}[{\cite[Proposition 2.2]{theleast}}]\label{proposition 4.5}
Let $Z_i\ge0$, $i\in[n]$, be random variables. Then for every $\varepsilon>0$, we have 
\[
\mathbb{P}\left(\frac{1}{n}\sum_{i=1}^n Z_i\le \varepsilon\right)\le \frac{2}{n}\sum_{i=1}^n \mathbb{P}(Z_i\le 2\varepsilon).
\]
\end{proposition}
Applying Proposition \ref{proposition 4.5} to $Z_i=(a_i/b_i)^2$ and using the bounds in (\ref{eq:10}), we deduce from (\ref{eq:8}) that
\begin{equation}\label{eq:11}
\mathbb{P}\left(\norm{(M_n^T)^{-1}x}_2\le\frac{\varepsilon\sqrt{n}}{\sqrt{2}t}\ \text{and}\ \mathcal{E}_0\right)\le4\sqrt{2}C\varepsilon+\frac{2C_1}{\sqrt{d}}+\frac{2}{t^2}.
\end{equation}
By using (\ref{eq:7}) and (\ref{eq:11}), we can complete our proof. For every $\varepsilon,s,t>0$,
\begin{align*}
\mathbb{P}\left(s_n(M_n)\le\frac{st}{\varepsilon}\sqrt{\frac{p}{n}}\right)&\ge\mathbb{P}\left(\norm{x}_2\le\frac{s\sqrt{p}}{\sqrt{2}}\ \mathrm{and}\ \norm{(M_n^T)^{-1}x}_2\ge\frac{\varepsilon\sqrt{n}}{\sqrt{2}t}\right)\\&\ge 1-\mathbb{P}\left(\norm{x}_2\ge\frac{s\sqrt{p}}{\sqrt{2}}\ \text{or}\ \norm{(M_n^T)^{-1}x}_2\le\frac{\varepsilon\sqrt{n}}{\sqrt{2}t}\right)\\&\ge1-\mathbb{P}\left(\norm{x}_2\ge\frac{s\sqrt{p}}{\sqrt{2}}\right)-\mathbb{P}\left(\norm{(M_n^T)^{-1}x}_2\le\frac{\varepsilon\sqrt{n}}{\sqrt{2}t}\ \text{and}\ \mathcal{E}_0\right)-\mathbb{P}(\mathcal{E}_0^c)\\&\ge1-\left(4\sqrt{2}C\varepsilon+\frac{2C_1}{\sqrt{d}}+\frac{2}{t^2}+\frac{2}{s^2}+2e^{-cn}\right)\\&\ge
1-C_p\left(\varepsilon+\frac{1}{\sqrt{d}}\right),
\end{align*}
where $C_p$ is a constant depending only on $p$, and we then choose $t=s=1/\sqrt{\varepsilon}$.
\printbibliography

\vspace{1.5cm}
\par\noindent
\textsc{Yuxiao Shi}
\\
\textit{Email address:} \texttt{yshi6821@gmail.com}

\par\noindent
Zhongtai Securities Institute for Financial Studies, Shandong University, Jinan, 250100, China.

\par\noindent
\textsc{Zeyan Song}
\\
\textit{Email address:} \texttt{zeyansong8@gmail.com}

\par\noindent
Zhongtai Securities Institute for Financial Studies, Shandong University, Jinan, 250100, China.

\par\noindent
\textsc{Kexin Yu}
\\
\textit{Email address:} \texttt{Kexinyu01@gmail.com}

\par\noindent
School of Mathematics, Shandong University, Jinan, 250100, China.

\par\noindent
\textsc{Tingzhou Yu} 
\\
\textit{Email address:} \texttt{tingzho1@ualberta.ca}

\par\noindent
Department of Mathematical and Statistical Sciences, University of Alberta, Edmonton, AB, T6G 2R3, Canada.

\end{document}